\documentclass[12pt,a4paper]{amsart}

\usepackage{amscd}
\usepackage{graphicx}
\usepackage{amssymb,amsfonts,latexsym}

\input xy
\xyoption {all}

\usepackage{geometry}
\newtheorem{thm}{Theorem}[section]

\newtheorem{cor}[thm]{Corollary}

\newtheorem{lem}[thm]{Lemma}

\newtheorem{prop}[thm]{Proposition}

\newtheorem{prob}[thm]{Problem}

\theoremstyle{definition}

\newtheorem{conj}[thm]{Conjecture}

\newtheorem{cri}[thm]{Criterion}
\newtheorem{defn}[thm]{Definition}

\newtheorem{exam}[thm]{Example}

\theoremstyle{remark}

\newtheorem{rem}[thm]{Remark}

\numberwithin{equation}{section}

\def\divides{{\,|\,}}

\def\ndivides{{\not|\,}}

\def\divides{{\,|\,}}

\def\ndivides{{\not|\,}}
\newcommand\card[1]{{\left|{#1}\right|}}

\newcommand{\mQ}{\mathbb Q}
\newcommand{\G}{\mathcal{G}}
\newcommand{\mH}{\mathcal{H}}

\newcommand{\ra}{\rightarrow}

\newcommand\lcm{{\operatorname{lcm}}}

\newcommand{\mG}{\mathfrak{G}}

\DeclareMathOperator\Gal{Gal}

\DeclareMathOperator\Lim{lim}

\begin{document}

\title[1]{Admissibility and field relations}%

\def\Tech{Department of Mathematics, Technion-Israel Institute of Technology, Haifa 32000, Israel}

\author{Danny Neftin}
\address{\Tech}
\email{neftind@tx.technion.ac.il}%

\begin{abstract}

Let $K$ be a number field. A finite group $G$ is called $K$-admissible if there exists a $G$-crossed product $K$-division algebra. $K$-admissibility has a necessary condition called $K$-preadmissibility that is known to be sufficient in many cases. It is a 20 year old open problem to determine whether two number fields $K$ and $L$ with different degrees over $\mQ$ can have the same admissible groups. We construct infinitely many pairs of number fields $(K,L)$ such that $K$ is a proper subfield of $L$ and $K$ and $L$ have the same preadmissible groups. This provides evidence for a negative answer to the problem. In particular, it follows from the construction that $K$ and $L$ have the same odd order admissible groups.
\end{abstract}

\maketitle

\section{introduction}

Equivalence relations between number fields are often used to determine the extent to which certain arithmetic properties of a field determine the field.
One example  is {\it arithmetic equivalence} (see \cite{Per} or \cite[Chap. III]{Kl}),
under which two number fields $K$ and $L$ are equivalent if they have the same Dedekind zeta function. Two arithmetically equivalent fields have the same $\mQ$-normal closure, degree over $\mQ$, inertia degrees of rational primes and a long list of other properties (see \cite[Chap. III, Theorem 1.4]{Kl}). In particular if $L/\mQ$ is Galois there is no number field $K$ different from $L$ that is arithmetically equivalent to $L$.

In \cite{Neu4}, Neukirch proved that if two number fields $K$ and $L$ have isomorphic absolute Galois groups
then they have the same $\mQ$-normal closure and asked if necessarily $K\cong L$. This was proved
independently by Ikeda, Iwasawa and Uchida (see \cite{Uch}).
In other words, the absolute Galois group of a number field determines it.

In \cite{Son}, Sonn asked an analogous question for crossed product division algebras and admissibility.
A finite group $G$ is {\it $K$-admissible} (here $K$ can be an arbitrary field) if there exists a $K$-central division algebra $D$ with maximal subfield $L$ such that $L/K$ is Galois with $\Gal(L/K)\cong G$. In such a case $D$ has the structure of a $G$-crossed product division algebra (see \cite[Chap. 4, \S 4]{Her}).

Two number fields $K$ and $L$ are {\it equivalent by admissibility}
if $K$ and $L$ have the same admissible groups.
In \cite{Son},  Sonn showed  that number fields which are equivalent by admissibility have the same $\mQ$-normal closure.
It is unknown whether two number fields which are equivalent by admissibility are necessarily isomorphic. In fact, even the following problem is open (see \cite{Son},\cite{Son1}):
\begin{prob}\label{degree.prob}Let $K$ and $L$ be two number fields that are equivalent by admissibility. Do $K$ and $L$ necessarily have the same degree over $\mQ$?
\end{prob}

So far, this problem was found to have an affirmative answer in several cases. Most notably, in \cite[Theorem 5]{Loc} Lochter showed that if in addition $[K:\mQ]$ is a prime or $[K:\mQ]=4$ then necessarily $[K:\mQ]=[L:\mQ]$.

In this paper we construct (infinitely many) examples of number fields $L$, Galois over $\mQ$, and proper subfields $K\subset L$ that are equivalent with respect to a property called preadmissibility (defined below) which is closely related to admissibility. In particular, we show that these number fields $K$ and $L$  have the same odd order admissible groups. The often indistinguishable behavior of admissibility and preadmissibility and these constructions lead us to conjecture
that Problem \ref{degree.prob} has a negative answer in the following form:
\begin{conj}\label{main.conj} There exists a number field $L$ that has a proper subfield $K$ such that $K$ and $L$ have the same admissible groups.
\end{conj}

Let us recall the definition of preadmissibility and its origin. In \cite{Sch}, Schacher gave the following realization criterion for admissibility over number fields:
\begin{cri}(Schacher)\label{schacher.cri} Let $G$ be a finite group and $K$ a number field. Then $G$ is $K$-admissible if and only if there exists a $G$-extension $L/K$ such that for every rational prime $p\divides |G|$, there are two primes $v_1 ,v_2$ of $K$ for which the decomposition group of $L/K$ at a prime over $v_i$ 
 contains a $p$-Sylow subgroup of $G$, for $i=1,2$.\end{cri}

Following Schacher's criterion, one can extract necessary local realization conditions for $K$-admissibility, namely:
\begin{defn} Let $K$ be a number field. A finite group $G$ is \emph{$K$-preadmissible} if for every $p\divides\card{G}$ there are two primes $v_1(p),v_2(p)$ and corresponding subgroups $G^{v_1(p)},G^{v_2(p)}\leq G$ such that $G^{v_i(p)}$ contains a $p$-Sylow subgroup of $G$ and is realizable over $K_{v_i(p)}$ for $i=1,2$, and the primes $v_i(p)$, $i=1,2$, $p\divides \card{G},$ are distinct\footnote{The condition on the primes to be distinct can be relaxed as in \cite{U}}. \end{defn}

It often happens that preadmissible groups are also admissible. For example, a theorem of Neukirch (see \cite[Corollary 2]{Neu}) implies that any group whose order is prime to the number of roots of unity in $K$ is $K$-preadmissible if and only if it is $K$-admissible.
The $\mQ$-preadmissible groups are those with metacyclic Sylow subgroups (see \cite{Nef}). It is known that many groups with metacyclic Sylow subgroups, including all such solvable groups, are $\mQ$-admissible (see e.g. \cite{AS},\cite{FS},\cite{Fei},\cite{FV},\cite{Son}). However there are such groups for which this remains unknown (see \cite[Theorem 2.3]{DS}).

We shall say that two number fields $K$ and $L$ are {\it equivalent by preadmissibility} if they have the same preadmissible groups.
Even though equivalence by admissibility is a complete mystery, in many cases equivalence by preadmissibility is within reach and allows one to study equivalence by admissibility with respect to various families of groups. Sonn's proof of \cite[Theorem 1]{Son} can be adapted to show that number fields that are equivalent by preadmissibility must have the same $\mQ$-normal closure (see Proposition \ref{pread-normal-closure.rem}). In particular $\mQ$ is equivalent by preadmissibility only to itself.
However, there are number fields $L$ that are Galois over $\mQ$ and have proper subfields $K$ that are preadmissibly equivalent to $L$ (see Corollary \ref{examples.cor}).
We use the following theorem to reduce this assertion to a group theoretical statement on split double cosets. For two subgroups $A,B$ of a finite group $\G$, a double coset $AxB$ is called {\it split} if $|AxB|=|A||B|$.

\begin{thm}\label{examples_characterization.thm}
Let $l$ be a prime and $\G$ an $l$-group. Let $L/\mQ$ be a $\G$-extension in which $l$ splits completely. Let $K$ be a subfield of $L$ and $\mH=\Gal(L/K)$. If for every subgroup $D\leq \G$ that appears as a decomposition group 
there are two split double cosets of the form $Dx\mH$, $x\in \G$, then $K$ and $L$ are equivalent by preadmissibility.  If $\G$ is non-metacyclic the converse also holds.
\end{thm}

Note that the requirement on $l$ to split completely is satisfied in many known realizations of $l$-groups including those of Scholz-Reichard (see \cite{Sco}) and Shafarevich (see \cite{Sha}).
By observing that extensions $L/\mQ$ as in Theorem \ref{examples_characterization.thm} are tamely ramified and hence have metacyclic decomposition groups and using the realizations of \cite{Sco}, \cite{Sha} and Neukirch's Theorem (see \cite{Neu}),
we obtain  the following group theoretic criterion:
\begin{cor}\label{group_criterion.cor} Let $l$ be a prime. Let $\G$ be an $l$-group and $\mH\leq \G$ a subgroup such that for every metacyclic subgroup $D\leq \G$ there are two split double cosets of the form $Dx\mH$. Then there is a $\G$-extension $L/\mQ$ such that $L$ and $K:=L^{\mH}$ are equivalent by preadmissibility and have the same odd order admissible groups. \end{cor}

The proofs of Theorem \ref{examples_characterization.thm} and Corollary \ref{group_criterion.cor} are given in Section \ref{equ_subfield.sec}. In Section \ref{seq.sec}, we give simple group theoretic
conditions for the construction of infinitely many pairs $(\G,\mH)$ as in Corollary \ref{group_criterion.cor}.
In Section \ref{sylow.sec}, we use the conditions from Section
\ref{seq.sec} to show that in Corollary \ref{group_criterion.cor}, $\G$ can be chosen to be an $l$-Sylow subgroup of the symmetric group $S_{l^n}$ and $\mH\leq \G$ a cyclic subgroup of order $l$, for any $n\geq 3$ and prime $l$. In particular Corollary \ref{group_criterion.cor} yields examples of preadmissibly equivalent fields $K, L$ with $[L:\mQ]>[K:\mQ]$.

For further insight into equivalence by admissibility, in Section \ref{art.sec} we compare the preadmissibility equivalence to other relations. Namely, we compare it to
arithmetic equivalence and {\it local isomorphism}, under which two number fields $K$ and $L$ are equivalent if they have isomorphic Adele rings (see \cite[Chap. VI, \S 2]{Kl}). We discuss the following implications diagram:
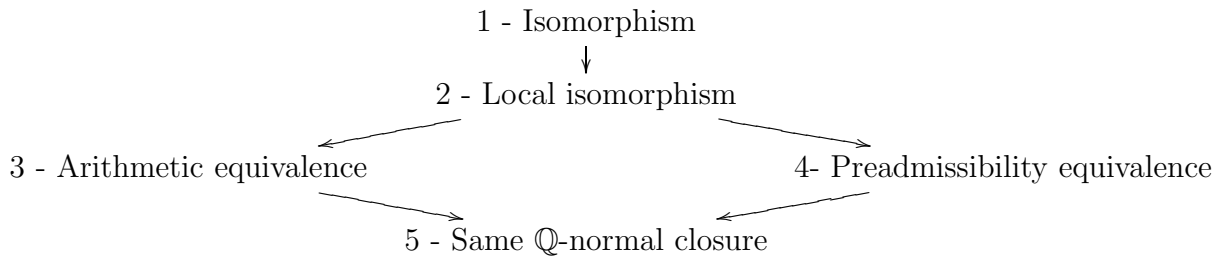
\begin{equation}\label{section1.1 - implications Diagram}
\xymatrix@R=10pt@C=3pt{
    & \mbox{1 - Isomorphism} \ar@{->}[d]  &
\\
    & \mbox{2 - Local isomorphism} \ar@{->}[ld]\ar@{->}[rd] &
\\
    \mbox{3 - Arithmetic equivalence \ar@{->}[dr]} &  & \mbox{4- Preadmissibility equivalence} \ar@{->}[dl]
\\
    &  \mbox{5 - Same $\mQ$-normal closure}  &
}
\end{equation}
and prove that every other implication that holds is a composition of these implications.

Note that an example for the non-implication $2\not\ra 1$ appears
in \cite{Kom2} and a different approach that yields a rich source
of examples is given in Example \ref{2 not implies 1.exam}. The
non-implication $3\not\ra 2$ is known by \cite{Kom1}.
It is unknown how equivalence by admissibility fits into Diagram (\ref{section1.1 - implications Diagram}) and whether it implies or is implied by the preadmissibility equivalence.

\subsection*{Acknowledgements}The paper is based on a  work of the author throughout his M.Sc. research under the supervision of Jack Sonn. The author would like to thank Prof. Sonn for reading several drafts of this paper, suggesting ways to improve it and for introducing the subject and the connections between prime decompositions and double cosets
to the author. The author would like to thank Lior Bary-Soroker
for his advice on the organization of this paper and for many
useful comments. The author would also like to thank the referee
for his/her helpful remarks.

\section{Equivalence by preadmissibility}

\subsection{Primes and double cosets}\label{basic.sec}

To understand equivalence by preadmissibility it is necessary to first understand how equivalent fields lie inside their $\mQ$-normal closure. For this, we first recall a well known connection between prime decompositions in subfields of Galois extensions and double cosets.

Let $\G$ be a finite group, $M/\mQ$ a $\G$-extension, $K$ a subfield of $M$ and $\mH:=\Gal(M/K)$. Let $p$ be a rational prime and let $v_1,...,v_k$ be the primes of $K$ lying above it. Assume the primes $v_1,...,v_k$ are ordered such that $[K_{v_i}:\mQ_p]\geq [K_{v_j}:\mQ_p]$ for $i\geq j$. We shall call the vector $([K_{v_1}:\mQ_p],...,[K_{v_k}:\mQ_p])$ the {\it local degree type} of $p$ in $K$.

The parallel notion in group theory is the double coset type. Let $Dx_1\mH,...,Dx_s\mH$ be the double cosets of $D\leq \G$ and $\mH$ in $\G$, ordered by decreasing cardinality: $|Dx_1\mH|\geq ...\geq |Dx_s\mH|$. We call the vector $(|Dx_1\mH|, ...,|Dx_s\mH|)$ the {\it double coset type} $(D,\mH)$. Denote by $(D,\mH)_k$ the $k$-th entry of the vector $(D,\mH)$.

Let $f(x)\in \mQ[x]$ be an irreducible polynomial, a root $\alpha$ of which generates $K/\mQ$. Then $f$ splits over $M$. Let $\alpha_1:=\alpha,\alpha_2,...,\alpha_r$ be the roots of $f$ in $M$.  Then $\G$ acts transitively on the set $R_f:=\{\alpha_1,...,\alpha_r\}$. This action is equivalent to the action of $\G$ on the set of left cosets $\G/\mH$. Let $w_1$ be a prime of $M$ dividing $v_1$ and $D:=D(M/\mQ,w_1)\leq \G$ the decomposition group of $w_1$ in $M/\mQ$. Then $D$ acts on the set $R_f$ which breaks into orbits under this action. Denote these orbits by $O_1,...,O_s$.
These orbits correspond to the orbits of the action of $D$ (as a subgroup of $\G$) on $\G/\mH$ which are $\{dx_i\mH|d\in D\}$, $i=1,...,s$.
Furthermore, the cardinality of $O_i$ equals the number of elements in the orbit of $x_i\mH$ which is $\frac{|Dx_i\mH|}{|\mH|}$, for any $i=1,...,s$.

The group $D$ is isomorphic to the Galois group of $M_{w_1}/\mQ_p$ which is a splitting field of $f$ over $\mQ_p$ with an isomorphism that preserves the action on $R_f$. Thus, $f$ factors over $\mQ_p$ into  $f(x)=f_1(x)...f_s(x)$ where $f_i$ is irreducible over $\mQ_p$ and the roots of $f_i$ are the elements of the orbit $O_i$. In particular $k=s$ and for all $i=1,\ldots,s,$ \begin{equation} \label{degrees equality}[K_{v_i}:\mQ_p]=\deg(f_i(x))=|O_i|=\frac{|Dx_i\mH|}{|\mH|}.\end{equation}

We arrive to the following well known description of prime decomposition in $K$:

\begin{lem}\label{double_coset_prime_splt_cor.lem}
Let $\G$ be a finite group and $\mH$ a subgroup of $\G$. Let $M$ be a $\G$-extension of $\mQ$ and $K=M^\mH$. Let $p$ be a rational prime and $v$ a prime of $M$ dividing $p$ with decomposition group $D:=D(M/\mQ,v)\leq \G$. Then the local degree type of $p$ in $K$ equals $\frac{1}{|\mH|}(D,\mH)$.
\end{lem}

Note that for a different prime $v'$ of $M$ dividing $p$, the decomposition group $D'=D(M/\mQ,v')$ is a conjugate of $D$ and hence $D$ and $D'$ have the same double coset type $(D',\mH)=(D,\mH)$.

Since for any two subgroups $A,B\leq \G$ we have $|AB|=\frac{|A||B|}{|A\cap B|}$, (\ref{degrees equality}) equals:
\begin{equation}\label{deg2.equ}\frac{|Dx_i\mH|}{|\mH|}=\frac{|x_i^{-1}Dx_i\mH|}{|\mH|}=\frac{|x_i^{-1}Dx_i||\mH|}{|x_i^{-1}Dx_i\cap \mH||\mH|}
=\frac{|D|}{|x_i^{-1}Dx_i\cap \mH|}.\end{equation} We shall use (\ref{degrees equality}) and (\ref{deg2.equ}) repeatedly throughout the text.

\subsection{The $\mQ$-normal closure}\label{normalclosure.subsection}
Similarly to \cite{Son} one has:
\begin{prop}\label{pread-normal-closure.rem} Two number fields that are equivalent by preadmissibility have the same $\mQ$-normal closure. \end{prop}
A proof for Proposition \ref{pread-normal-closure.rem}  can be obtained by adjusting the proof of \cite[Theorem~1]{Son} to preadmissibly equivalent fields and using a weak version of the Grunwald-Wang theorem for odd order groups (see \cite{Wan} or \cite[Chap. IX, \S 2, Theorem 9.2.8]{NSW}):
\begin{thm}\emph{(Grunwald-Wang)}\label{GW.thm} Let $A$ be an abelian group of odd order. Let $K$ be a number field and $S$ a finite set of primes of $K$. For every prime $v\in S$ fix an extension $L^{(v)}/K_v$ whose Galois group is isomorphic to a subgroup of $A$. 
Then there is an $A$-extension $L/K$ for which $L_v=L^{(v)}$ for all $v\in S$.
\end{thm}
In particular it follows from Criterion \ref{schacher.cri} that an odd order abelian group $A$ is $K$-admissible if and only if $A$ is $K$-preadmissible.

The proof of \cite[Theorem 1]{Son} can be adjusted as follows: in all cases choose $p$ to be an odd prime, in Cases 1 and 2.1 one may pick $B$ to be $C_p\wr C_p$  and  in Case 2.2 pick $A$ to be $(C_p)^3$. Letting $F$ be any of the fields $K$ and $L$ in \cite[Theorem 1]{Son}, one observes that both $A$ and $B$ are $F$-admissible if and only if they are $F$-preadmissible. Indeed by Theorem \ref{GW.thm}, $A$ is $F$-admissible if and only if it is $F$-preadmissible over any number field $F$ and by \cite[Theorems 3.3 and 5.10(b)]{Sal} $B$ admits the same property. Therefore making this choice of $A$ and $B$ separates $K$ and $L$ by preadmissibility in the same way \cite[Theorem 1]{Son} separates $K$ and $L$ by admissibility.

Note that having the same $\mQ$-normal closure is considered a weak arithmetic equivalence and by \cite[Chap. II, Corollary 1.6.b]{Kl} is characterized by the set of rational primes that split completely.

The following proposition presents further properties of preadmissibly equivalent fields.
We shall say a prime $p$ decomposes in a number field $F$ if there are at least two primes of $F$ that divide $p$.
\begin{prop}\label{preadmissibility_implications.prop}
Let $K$ and $L$ be two number fields that have the same $\mQ$-normal closure $M$ and $\mH=\Gal(M/K),\mH'=\Gal(M/L)$. Let $p$ be a rational prime, $v$ a prime of $M$ dividing $p$ and $D=D(M/\mQ,v)$. If $K$ and $L$ are equivalent by preadmissibility then:

\emph{(1)} $p$ decomposes in $K$ if and only if $p$ decomposes in $L$,

\emph{(2)}
if $p$ decomposes in $K$ and $L$ then $\frac{(D,\mH)_2}{|\mH|}=\frac{(D,\mH')_2}{|\mH'|}$.
\end{prop}
The idea behind the proof of Proposition \ref{preadmissibility_implications.prop} lies in \cite[Theorems 1 and 2]{Son}. A similar proposition, with the assumptions that $K$ and $L$ are equivalent by admissibility and $p$ is odd, appears in \cite{Loc}  without proof. Our proof uses the following lemma which is part of \cite[Theorem 28]{Lid2}:
\begin{lem}\label{preadmissibility_implies_metacyclic}
Let $K$ be a number field and $p$ a rational prime that does not decompose in $K$. If $G$ is a $K$-preadmissible group then the $p$-Sylow subgroups of $G$ are metacyclic.
\end{lem}
\begin{rem}\label{Lid_observation}In \cite{Lid2}, $G$ is assumed to be $K$-admissible but the proof only uses preadmissibility (see also \cite{Nef}).
What is actually proved in \cite{Lid2} is the following observation: If $G$ is $K$-preadmissible with respect to a set of primes $\{v_i(p)|i=1,2,p\divides |G|\}$ and $p_0\divides |G|$ is a prime for which $v_i(p_0)$ does not divide $p_0$ for some $i=1,2$, then any $p_0$-Sylow subgroup of $G$ is metacyclic.
\end{rem}
\begin{proof}[Proof of Proposition \ref{preadmissibility_implications.prop}]
(1) Assume on the contrary $p$ decomposes in $L$ but not in $K$ and let
$v_1,v_2$ be two primes of $L$ that divide $p$. Then by Lemma \ref{preadmissibility_implies_metacyclic} any $K$-preadmissible $p$-group is metacyclic. Let $G=C_p\wr C_p$ for odd $p$ and $G=C_2\wr C_4$ for $p=2$. Then $G$ is not metacyclic and hence not $K$-preadmissible. We shall show that $G$ is $L$-preadmissible which leads to a contradiction.

For this, it suffices to prove that $G$ is realizable over $L_{v_1}$ and $L_{v_2}$. Let $k$ be a $p$-adic field, $M_p(k)$ the maximal pro-$p$ extension of $k$ and denote by $\mG_k$ the Galois group $\Gal(M_p(k)/k)$. For odd $p$, $\mG_{\mQ_p}$ is the free pro-$p$ group on two generators (see \cite[\S 2.5.6]{Ser2}). By the Nielsen-Schreier theorem for pro-$p$ groups (see \cite{BNW}) the subgroup $$\Gal(M_p(\mQ_p)k/k)\cong \Gal(M_p(\mQ_p)/M_p(\mQ_p)\cap k)\leq \mG_{\mQ_p}$$ is a free pro-$p$ group of rank $[\mG_{\mQ_p}:\Gal(M_p(\mQ_p)/M_p(\mQ_p)\cap k)]+1\geq 2$  and hence $G$ is a quotient of it. Thus, for odd $p$, $G$ is realizable over $k=L_{v_1},L_{v_2}$.

Now let $p=2$ and $k$ be a $2$-adic field of degree $n=[k:\mQ_2]$.
If $n$ is even then $\mG_k$ has one of the following presentations of pro-$2$ groups (see e.g. \cite[\S 2.5.6]{Ser2}):
\begin{equation}\label{first_presentation.equ}
\langle x_1,\ldots,x_{n+2} \mid x_1^{q}[x_1,x_2] \cdots [x_{n+1},x_{n+2}]=1 \rangle, 
\end{equation}
\begin{equation}\label{f_1_presentation.equ} \langle x_1,\ldots,x_{n+2} \mid x_1^{2}[x_1,x_2]x_3^{2^f}[x_3,x_4] \cdots [x_{n+1},x_{n+2}]=1 \rangle, \end{equation}
\begin{equation}\label{f_2_presentation.equ} \langle x_1,\ldots,x_{n+2} \mid x_1^{2+2^f}[x_1,x_2] \cdots [x_{n+1},x_{n+2}]=1 \rangle. \end{equation}
where $q\geq 2$ is the number of $2$-power roots of unity in $k$ and $f\geq 2$ an integer. For $n$ odd, $\mG_k$ has the presentation:
\begin{equation}\label{last_presentation.equ} \langle x_1,\ldots,x_{n+2} \mid x_1^2 x_2^4 [x_2,x_3] \cdots [x_{n+1},x_{n+2}]=1 \rangle . \end{equation}
Let $\langle a\rangle =C_2, \langle b\rangle = C_4$. From each of the presentations (\ref{first_presentation.equ})-(\ref{last_presentation.equ}), there is an epimorphism from $\mG_k$ to $G=\langle a \rangle \wr \langle b\rangle= \langle a\rangle^{4}\rtimes \langle b\rangle$ simply by sending $x_1$  to one of the conjugates of $a$, $x_2\ra 1,x_3\ra b$ and $x_i\ra 1$ for $i\geq 4$. Thus,  $C_2\wr C_4$ is realizable over any $2$-adic field. It follows that for any prime $p$,
$G$ is realizable over $L_{v_1},L_{v_2}$ and hence is $L$-preadmissible.

(2) Let $v$ be a prime of $K$ dividing $p$. By local class field theory the Galois group $A_v:=\Gal(K_{v,ab,p}/K_v)$, where $K_{v,ab,p}$ is the maximal abelian pro-$p$ extension of $K_v$, is isomorphic to the pro-$p$ completion of the group $K_v^*$. Thus the rank $r_v$ of the maximal free abelian quotient of $A_v$ is $r_v=[K_v:\mQ_p]+1$ (see e.g. \cite[Chap. 14, \S 6]{Ser3}).
Let $p^{n_0}$ be larger then the exponent of the torsion part of $A_v$ for all primes $v$ of $K$ and $L$ that divide $p$. Then, for a prime $v$ of $K$ dividing $p$, the group $(C_{p^{n_0}})^N$ is realizable over $K_v$  if and only if  $N\leq r_v$. Let $v_2$ (resp. $w_2$) be a prime of $K$ (resp. $L$) dividing $p$ such that $[K_{v_2}:\mQ_p]$ (resp. $[L_{w_2}:\mQ_p]$)\ is second in the local degree type of $p$ in $K$ (resp. in $L$).

Assume on the contrary that $\frac{(D,\mH)_2}{|\mH|} > \frac{(D,\mH')_2}{|\mH'|}$.  By (\ref{degrees equality}), $[K_{v_2}:\mQ_p]=\frac{(D,\mH)_2}{|\mH|}$ and $[L_{w_2}:\mQ_p]=\frac{(D,\mH')_2}{|\mH'|}$. Thus, $r_{v_2}\geq [K_{v_2}:\mQ_p] +1\geq [L_{w_2}:\mQ_p]+2\geq 3$. The group $G=(C_{p^{n_0}})^{r_{v_2}}$ is therefore not metacyclic and hence realizable only over completions at primes dividing $p$. The above discussion shows that $G$ is realizable over two completions of $K$ but over at most one of $L$. Thus, $G$ is $K$-preadmissible  but not $L$-preadmissible, contradiction.
\end{proof}
Note that by  Theorem \ref{GW.thm}, for odd $p$ the group $G$ in the proof of Proposition \ref{preadmissibility_implications.prop}.(2) is in fact $K$-admissible.
\subsection{Equivalent subfields}\label{equ_subfield.sec}
Let $K$ and $L$ be number fields that are equivalent by preadmissibility and let $M$ be their $\mQ$-normal closure. Let us now assume further that  $L$ is Galois over $\mQ$. Then $L=M$ and $K$ is a subfield of $L$. As we shall now see, in this case Proposition \ref{preadmissibility_implications.prop} implies the second assertion of Theorem \ref{examples_characterization.thm} using Corollary \ref{equivalence_implies.cor} below.

For two subgroups $A,B$ of a finite group $\G$, denote by $S(A,B)$ the number of distinct split double cosets of $A,B$ in $\G$. By (\ref{deg2.equ}), $|AxB|=\frac{|A||B|}{|x^{-1}Ax\cap B|}$ and hence $AxB$ is a split double coset if and only if $x^{-1}Ax\cap B=1$.
\begin{cor}\label{equivalence_implies.cor} Let $\G$ be a finite group and $L/\mQ$ a $\G$-extension in which every rational prime decomposes. Let $K$ be a subfield of $L$ that is equivalent to $L$ by preadmissibility and let $\mH=\Gal(L/K)$.  Then $S(D,\mH)>1$ for any $D\leq \G$ that appears as a decomposition group.
\end{cor}
\begin{proof}  Let $p$ be a rational prime and $v_0$ a prime of $L$ dividing $p$. All primes $v$ of $L$ dividing $p$ have the same degree $[L_v:\mQ_p]=|D|$, where $D=D(L/\mQ,v_0)$. Since $p$ decomposes in $L$, and $K$ is equivalent by preadmissibility to $L$,
Proposition \ref{preadmissibility_implications.prop} implies that $p$ decomposes in $K$ and $\frac{(D,\mH)_2}{|\mH|}=|D|$. We deduce that $(D,\mH)_1=(D,\mH)_2=|D||\mH|$ and hence $S(D,\mH)>1$.
\end{proof}

If $\G$ is an $l$-group and $L/\mQ$ is a $\G$-extension in which $l$ splits completely then the ramification in $L/\mQ$ is tame and  all decomposition groups are metacyclic.
If in addition $\G$ is a non-metacyclic 
then every rational prime decomposes in $L$ and Corollary \ref{equivalence_implies.cor} applies. In particular Corollary \ref{equivalence_implies.cor} implies the second assertion of Theorem \ref{examples_characterization.thm} (the ``converse" part).

Note that  by the Chebotarev density theorem every cyclic subgroup appears as a decomposition group and hence Corollary \ref{equivalence_implies.cor} implies that $S(C,\mH)>1$ holds for every cyclic subgroup $C\leq \G$.

We shall now prove the first assertion of Theorem \ref{examples_characterization.thm} (the ``forward" part). The following proposition proves a part of this implication:
\begin{prop}\label{preadmissibility_under_extension.prop}
Let $l$ be a prime and $\G$ an $l$-group. Let $L/\mQ$ be a $\G$-extension in which $l$ splits completely and every rational prime decomposes. Let $K$ be a subfield of $L$.
Then every $K$-preadmissible group is also $L$-preadmissible.
\end{prop}

Note that the condition $S(D,\mH)>1$ for every subgroup $D\leq \G$ that appears as a decomposition group in $L/\mQ$ insures that every rational prime decomposes in $K$ and hence in $L$. Therefore, Proposition \ref{preadmissibility_under_extension.prop} shows that under the conditions of the first assertion of Theorem \ref{examples_characterization.thm}, any $K$-preadmissible group is also $L$-preadmissible. Our proof of Proposition \ref{preadmissibility_under_extension.prop} requires two lemmas.
\begin{lem}\label{section3- Remark - local extension of realizations} Let $G$ be a finite group. Let $L/K$ be an extension of $p$-adic fields such that $p\ndivides [L:K]$. Assume there is a subgroup  $G_1 \leq G$ that contains a $p$-Sylow subgroup of $G$ and is realizable
over $K$. Then there is a subgroup $G_2 \leq G_1$ that contains a $p$-Sylow subgroup of $G$ and is realizable over $L$.
\end{lem}
\begin{proof} Let $F/K$ be a $G_1$-extension. Let $G_2:=\Gal(F/F\cap L)\leq G_1\leq G$.
Then $G_2$ is isomorphic to $\Gal(FL/L)$ and hence realizable over $L$.  But as $p\ndivides [F\cap L:K]=[G_1:G_2]$, $G_2$ must also contain a $p$-Sylow subgroup of $G$.
\end{proof}

We shall use the following lemma to pass from tame realizations to wild realizations:
\begin{lem}\label{section3- Lemma - from tame to wild} Let $G$ be a metacyclic $p$-group and $k$ a $p$-adic field. Then $G$ is realizable over $k$.
\end{lem}

\begin{proof} Let $k\not=\mQ_2$ be a $p$-adic field, $n:=[k:\mQ_p]$ and $q$ the number of $p$-power roots of unity in $k$. At first assume $q>2$. Let $M_p(k)$ be the maximal pro-$p$ extension of $k$ and $\mG_k:=\Gal(M_p(k)/k)$. By \cite{Dem} (see also \cite[\S 2.5.6]{Ser2}), $\mG_k$ has a pro-$p$ presentation as in (\ref{first_presentation.equ}).
In such a case $n\geq 2$ and $\mG_k$ has an epimorphism onto the free pro-$p$ group $F_p(2)=\langle f_1,f_2\rangle$ on the $2$ generators $f_1,f_2$ which can be obtained simply by sending $x_2\ra f_1$, $x_4\ra f_2$ and $x_i\ra 1$ for every $i\not= 2,4$.  In particular, $G$ is realizable over $k$.

Now assume $q\leq 2$. If $p\not=2$,  then $q=1$ and by \cite{Shf} (see also \cite[\S 2.5.6]{Ser2}) $\mG_k$ is the free pro-$p$ group $F_p(n+1)$. Since $n\geq 1$, in this case as well $G$ is realizable over $k$.

Assume $p=q=2$ and $n\geq 2$,  then $\mG_k$ has one of the pro-$p$ presentations given by (\ref{f_1_presentation.equ}) and (\ref{f_2_presentation.equ}) if $n$ is even and (\ref{last_presentation.equ}) if $n$ is odd.
In the cases described in (\ref{f_1_presentation.equ}) and (\ref{f_2_presentation.equ}), $F_2(2)=\langle f_1,f_2\rangle$ is  an epimorphic image of $\mG_k$ by again sending $x_2\ra f_1,x_4\ra f_2$ and $x_i\ra 1$ for every $i\not=2,4$. If $n> 1$ is odd ($n\geq 3$) then $\mG_k$ has Presentation  (\ref{last_presentation.equ}) and hence admits an epimorphism onto  $F_2(2)=\langle f_1,f_2\rangle$ by sending $x_3\ra f_1,x_5\ra f_2$ and $x_i\ra 1$ for $i\not= 3,5$.

We are left with the case $p=2$, $k=\mQ_2$  which is covered in \cite{U}. \end{proof}

Let us turn back to prove Proposition \ref{preadmissibility_under_extension.prop}:
\begin{proof}%
Let $G$ be a  $K$-preadmissible group and $p\divides |G|$. There are two primes $v_1(p),v_2(p)$ of $K$ and corresponding subgroups $G^{v_1(p)},G^{v_2(p)}$ such that $G^{v_i(p)}$ is realizable over $K_{v_i(p)}$ and contains a $p$-Sylow subgroup of $G$, for $i=1,2$. For every $p$, we shall choose two primes $w_1(p),w_2(p)$ of $L$, and corresponding subgroups $G^{w_i(p)}\leq G$ such that:

(1) $G^{w_i(p)}$ contains a $p$-Sylow subgroup of $G$,

(2) $G^{w_i(p)}$ is realizable over $L_{w_i(p)}$,

(3) $w_1(p)\not=w_2(p)$ and $w_i(p)|p$,

 for all  $i=1,2$ and $p\divides |G|$. In particular,  $w_i(p)\not=w_j(q)$ for any $i,j\in\{1,2\}$ and $p\not=q$. Such a choice of primes and corresponding subgroups will show that $G$ is $L$-preadmissible.

Let $p\divides |G|$. If one of $v_i(p)$, $i=1,2$, does not divide $p$ then by Remark \ref{Lid_observation}, $G(p)$ is metacyclic and hence by Lemma \ref{section3- Lemma - from tame to wild}, realizable over any completion $L_w$ for any prime $w$ of $L$ that divides $p$. As $p$ decomposes in $L$,  we can choose both $w_1(p),w_2(p)$ to be any distinct primes of $L$ dividing $p$ and $G^{w_i(p)}:=G(p)$, $i=1,2$. So let us assume $v_1(p),v_2(p)|p$ and split our proof into two cases: $p=l$ and $p\not=l$.

Case $p=l$:  for every prime $w$ of $L$ dividing $l$ with restriction $v$ to $K$, we have $K_{v}\cong L_w\cong \mQ_l$.
In particular, if $l$ divides both $v_1(l),v_2(l)$ then $G^{v_i(l)}$ is realizable over $K_{v_i(l)}\cong L_{w_i(l)}$,  for any prime $w_i(l)$ of $L$ that divides $v_i(l)$, $i=1,2$. Thus by setting $G^{w_i(l)}:=G^{v_i(l)}$ for $i=1,2$ conditions (1)-(3) are satisfied for $p=l$.

Case $p\not= l$:
By Lemma \ref{section3- Remark - local extension of realizations}, for any $w_1(p)| v_1(p),w_2(p)| v_2(p)$ primes of $L$, there are two subgroups $G^{w_1(p)}\leq G^{v_1(p)},G^{w_2(p)}\leq G^{v_2(p)}$, each containing a $p$-Sylow subgroup of $G$, such that $G^{w_i(p)}$ is realizable over $L_{w_i(p)}$, for $i=1,2$.

The primes $w_i(p)$, and the corresponding subgroups $G^{w_i(p)}\leq G$ for $i=1,2$, $p\divides |G|$ were chosen so that conditions (1)-(3) hold and therefore $G$ is $L$-preadmissible.
\end{proof}

Lemma \ref{section3- Remark - local extension of realizations}  can also be used to extend admissibility from $K$ to $L$. However certain restrictions are required:

\begin{cor} Let $l$ be a prime and $\G$ an $l$-group. Let $L/\mQ$ be a $\G$-extension in which $l$ splits completely and $K$ a subfield of $L$. Then any group $G$ that is $K$-admissible and has no metacyclic Sylow subgroups is also $L$-admissible.
\end{cor}

\begin{proof} As $G$ is $K$-admissible there is a $G$-extension $F/K$ such that for every $p\divides |G|$ there are two primes $v_1(p),v_2(p)$ of $K$, the decomposition groups of which in $F/K$ contain a $p$-Sylow subgroup of $G$.

We claim $FL/L$ is a $G$-extension that satisfies Schacher's Criterion (Criterion \ref{schacher.cri}). The decomposition groups of $v_i(p)$ in $F/K$ are not metacyclic and hence $v_i(p)|p$ for all $p\divides |G|,i=1,2$. For every $i=1,2$, $p\divides |G|$, choose a prime $w_i(p)$ of $L$ that divides $v_i(p)$. By Lemma \ref{section3- Remark - local extension of realizations}, for $p\not=l$ that divides $|G|$ and $i=1,2,$ the decomposition groups of $w_i(p)$ in $FL/L$ contain a $p$-Sylow subgroup of $G$. When $p=l$, $K_{v_i}\cong L_{w_i}$ and hence the decomposition groups of $v_i$ in $F/K$ are the same as those of $w_i$ in $FL/L$, for $i=1,2$.

For all $p||G|$, we have a $p$-Sylow subgroup of $G$ that is contained in a decomposition group of $\Gal(FL/L)$ and hence in $\Gal(FL/L)$. In particular if $p^k||G|$ then $p^k||\Gal(FL/L)|$ and hence $|G|$ divides $|\Gal(FL/L)|$. Since $\Gal(FL/L)$ is isomorphic to a subgroup of $G$ we have $\Gal(FL/L)\cong G$. It follows that $FL/L$ is a $G$-extension and for every $p||G|$ there are two primes $w_1(p),w_2(p)$ of $L$ whose decomposition groups contain $p$-Sylow subgroups of $G$. This proves the claim and hence that $G$ is $L$-admissible.
\end{proof}

To prove Theorem \ref{examples_characterization.thm} we are left to prove:
\begin{prop}\label{L-pread_implies_K-pread.prop}
Let $l$ be a prime and $\G$ an $l$-group.
Let $L/\mQ$ be a $\G$-extension in which $l$ splits completely. Let $\mH\leq \G$ be a subgroup for which $S(D,\mH)>1$ for every subgroup $D\leq \G$ that appears as a decomposition group. Then any $L$-preadmissible group is also $K=L^\mH$ preadmissible.
\end{prop}
\begin{proof}
Let $G$ be an $L$-preadmissible group. For every $p\divides |G|$, there are two primes $w_1(p),w_2(p)$ of $L$ and corresponding subgroups $G^{w_1(p)},G^{w_2(p)},$ such that $G^{w_i(p)}$ is realizable over $L_{w_i(p)}$ and contains a $p$-Sylow subgroup of $G$.

Similarly to the proof of Proposition \ref{preadmissibility_under_extension.prop}, we show that the primes $w_i(p)$ can be chosen such that $w_i(p)|p$ for every $i=1,2$ and $p\divides |G|$.
If $w_i(p)$ does not divide $p$ by Remark~\ref{Lid_observation}, any $p$-Sylow subgroup $G(p)$ is metacyclic. In such a case, by Lemma \ref{section3- Lemma - from tame to wild}, $G(p)$ is realizable over $L_w$ for any prime $w$ that divides $p$. We replace every $w_i(p)$ that does not divide $p$ by a prime $w$ that divides $p$ and is different from $w_j(p)$, $j=1,2$, and set $G^w:=G(p)$. We obtain a set of primes $\{w_i(p)| i=1,2, p\divides |G|\}$ and corresponding subgroups $G^{w_i(p)}$ such that for every $i=1,2$ and $p\divides |G|$:

(1) $w_i(p)|p$,

(2) $G^{w_i(p)}$ is realizable over $L_{w_i(p)}$,

(3) $G^{w_i(p)}$ contains a $p$-Sylow subgroup of $G$.

Fix a rational prime $p\divides |G|$, a prime $w_0$ of $L$ dividing $p$ and set $D=D(L/\mQ,w_0)$. Since $L/\mQ$ is Galois, for any $i=1,2$, $G^{w_i(p)}$ is realizable over $L_{w_i(p)}$ and hence over $L_w$ for any prime $w$ of $L$ that divides $p$.  By the correspondence in Section \ref{basic.sec} the existence of two split double cosets in $(D,\mH)$ implies by (\ref{degrees equality}) that there are two primes $v_1(p),v_2(p)$ of $K$ for which $[K_{v_i(p)}:\mQ_p]=|D|$. Thus, there is a unique prime $w_i$ of $L$ that divides $v_i(p)$. For this prime we have $L_{w_i}\cong K_{v_i(p)}$. Therefore $G^{w_i(p)}$ is realizable over $K_{v_i(p)}$ for $i=1,2$ and $p\divides |G|$, which shows that $G$ is $K$-preadmissible.
\end{proof}

Theorem \ref{examples_characterization.thm} follows and we can now also prove Corollary \ref{group_criterion.cor}:

\begin{proof}
Let $S$ be the set of rational primes $p$ for which  $\mQ(e^{\frac{2\pi i}{p}})$ is contained in a $\G$-extension. Since such primes $p$ must satisfy $[\mQ(e^{\frac{2\pi i}{p}}):\mQ]=p-1\leq |\G|$, $S$ is finite.

By \cite{Sco} for odd $l$ (see also \cite[Chap. 2]{Ser92}) and \cite{Sha} for $l=2$ (see also \cite[Chap.~IX, \S 6]{NSW}), there are $\G$-extensions of $\mQ$ in which $l$ splits completely. Furthermore, in \cite[Chap. 2]{Ser92} and \cite[Chap. IX, \S 6]{NSW}, each ramified prime is chosen from an infinite set and hence there is a $\G$-extension $L/\mQ$ in which $l$-splits completely and the primes of $S$ are unramified.
Since any extension by roots of unity that is contained in a $\G$-extension must be ramified at some prime of $S$, the only roots of unity in $L$ are $\{1,-1\}$.

Let $K=L^\mH$. By Theorem \ref{examples_characterization.thm}, $K$ and $L$ are equivalent by preadmissibility. Since there are no odd order roots of unity in $K$ and $L$, we may apply Neukirch's Theorem (see \cite[Corollary 2]{Neu}) and deduce that every odd order group is $K$-preadmissible (resp. $L$-preadmissible) if and only if it is $K$-admissible (resp. $L$-admissible). But by Theorem~\ref{examples_characterization.thm}, $K$ and $L$ are equivalent by preadmissibility and hence every odd order group is $K$-admissible if and only if it is $L$-admissible.
\end{proof}
\section{Constructions}
\subsection{Sequences of $l$-groups}\label{seq.sec}

According to Corollary \ref{group_criterion.cor}, in order to construct preadmissibly equivalent fields with different degrees over $\mQ$, it suffices to find pairs of $l$-groups $\mH< \G$  that satisfy the condition: $S(D,\mH)>1$ for any metacyclic subgroup $D\leq \G$. We shall now provide a criterion on sequences of pairs $\mH< \G$ that guarantees that this condition is satisfied.

Fix a rational prime $l$. Let $(\G_n)_{n\in N}$ denote a sequence of $l$-groups such that $\G_n\leq \G_{n+1}$ and $\G_n\leq S_{l^n}$ for every $n\in \mathbb{N}$.  Let $\alpha$ be an element of order $l$ in  $\G_k$ for some $k\in\mathbb{N}$ and $\mH=\langle \alpha\rangle$.

\begin{prop}\label{series of groups} Let $d_n$ denote the maximal order of a metacyclic subgroup of $\G_n$, $n\in \mathbb{N}$. Assume the sequence $(\G_n)_{n\in\mathbb{N}}$ satisfies:

\emph{(1)} $\Lim_{n\ra \infty}\frac{|\G_n|}{d_n}=\infty$,

\emph{(2)} the element $\alpha$ has infinitely many conjugates in $\cup_{n\in\mathbb{N}}\G_n$.

Then there is an $N$ such that for every $n\geq N$ and every
metacyclic subgroup $D\leq \G_n$ we have $S(D,\mH)>1$.
\end{prop}

\begin{rem} Let $c_n$ denote the maximal order of an element in $\G_n$. Then $d_n\leq c_n^2$. Therefore the condition $\Lim_{n\ra\infty}\frac{|\G_n|}{c_n^2}=\infty$ suffices in order for $\G$ to satisfy $(1)$. Let $O_{\G_n}(\alpha)$ denote the orbit of $\alpha$ under conjugation in $\G_n$. Condition $(2)$ can also be stated as $\Lim_{n\ra \infty} |O_{\G_n}(\alpha)|=\infty$.
\end{rem}
It follows from Proposition \ref{series of groups} and Corollary \ref{group_criterion.cor} that:
\begin{cor}\label{examples.cor} Let $(\G_n)_{n=1}^\infty$ and $\mH$ be as in Proposition \ref{series of groups}. Then there is an $N$ such that for every $n\geq N$ there is a $\G_n$-extension $L/\mQ$ for which $K=L^\mH$ and $L$ are equivalent by preadmissibility and have the same odd order admissible groups.
\end{cor}
In order to prove Proposition \ref{series of groups} we first obtain a bound on the number of occurrences of a given cycle structure (a cycle structure is also often referred to as a partition) in an embedding of a metacyclic group in $S_n$.

Let $S_\infty$ be the group of all permutations of $\mathbb{N}$ that fix all but finitely many elements. Identify $S_n$ with the subgroup of $S_\infty$ that fixes all elements in $\mathbb{N}\setminus\{1,..,n\}$.
Note that $S_\infty$ can also be viewed as $S_\infty=\bigcup_{n\in\mathbb{N}}S_n$. Any element $\sigma\in S_\infty$ has a cycle structure $x=p(\sigma)$ which is a vector $(a_1,a_2,...)$ with $a_i\geq a_{i+1}$ such that $a_i=1$ for all $i\in \mathbb{N}$ but a finite number of $i$'s and for which $\sigma$ is a product of disjoint cycles $(\sigma_i)_{i\in \mathbb{N}}$ where $\sigma_i$ is an $a_i$-cycle for all $i\in \mathbb{N}$.
The order of $\sigma$ in $S_\infty$ is $\lcm_{i\in \mathbb{N}}(a_i)$ and hence depends only on the cycle structure $x$. We denote it by $o(x)$. 
Denote by $l(x)$ the length of $x$: $l(x):=\sum_{a_i\not=1}a_i$.

\begin{defn}\label{cyclicity_level.def} Let $G$ be a finite solvable group. Then there is a sequence $1=H_0\lhd H_1 \lhd \cdots \lhd H_k=G$ such that $H_i$ is normal in $H_{i+1}$ and $H_{i+1}/H_i$ is cyclic. The cyclicity level of $G$ is defined to be the minimal number $k$ for which such a sequence exists.
\end{defn}
\begin{lem}\label{existense of fx.prop}
Fix a number $k\in \mathbb{N}$ and some cycle structure $x$ in $S_\infty$. Then there is a number $b\in \mathbb{N}$ such that for every
solvable group $G$ of cyclicity level $k$ and every embedding $\phi:G\hookrightarrow S_\infty$ there are at most $b$ elements with cycle structure $x$ in $\phi(G)$.
\end{lem}
\begin{proof}
By induction on $k$. The case $k=0$ is trivial, for example take
$b=1$. Assume by induction that every group of cyclicity level
$<k$ has at most $e_y$ elements of cycle structure $y$ (in any
embedding). We fix $\phi$ and show that the number of elements of
cycle structure $x$ in $\phi(G)$ is bounded by a bound that
depends only on $k$ and $x$. We shall
identify $G$ with $\phi(G)$.

Let $H$ be a normal subgroup of $G$ of cyclicity level $k-1$ such that $C:=G/H$ is cyclic and let $\tau\in G$ be an element for which $\langle \tau H\rangle=C$. 
Let $u\in G$ be an element of cycle structure $x$. The order of the coset $uH$ in $C$ divides the
order of $u$ which is $o(x)$. As $C$ is cyclic it contains at most $o(x)$ elements of order
dividing $o(x)$ and hence there at most $o(x)$ cosets in $C$ that contain an element of cycle structure $x$.

It remains to bound the number of elements with cycle structure
$x$ in a given coset $uH$. Let $v$ be another element in $uH$ with
cycle structure $x$. The element $uv^{-1}$ is in $H$ and has
length $l(uv^{-1})\leq 2l(x)$. For every cycle structure $y$, with
length $l(y)\leq 2l(x)$ (clearly there are only finitely many
such) there are at most $e_y$ elements with cycle structure $y$ in
$H$ and hence $H$ contains at most $\sum_{\{y:l(y)\leq 2l(x)\}}
e_y$ elements with a cycle structure of length $\leq 2l(x)$. The
map $uH\ra H$ that sends $v\in uH$ to $u^{-1}v\in H$ is injective
and therefore the coset $uH$ contains at most $\sum_{\{y:l(y)\leq
2l(x)\}} e_y$ elements with cycle structure $x$. Summing over the
cosets of $C$ whose order divide $o(x)$ we get:
$$ |\{\sigma\in G|p(\sigma)=x\}|\leq b:=o(x)\sum_{\{y:l(y)\leq 2l(x)\}} e_y.$$
\end{proof}

For $k=2$, we have:
\begin{cor}\label{k-cyclicity and amount of a fixed type lemma}
Let $x$ be any cycle structure. There exists a number $b\in \mathbb{N}$
for which every metacyclic subgroup $D\hookrightarrow S_\infty$ contains at most $b$ elements with cycle structure $x$.
\end{cor}

\begin{exam} The maximal number of transpositions in an abelian $2$-group of rank $r$ is $r$.
\end{exam}
\begin{exam}\label{transposition_metacyclic.exm} The maximal number of transpositions
in a metacyclic group is $4$. The subgroup $\langle (123)(45),(12)
\rangle$ of $S_5$ is a metacyclic group with $4$ transpositions,
namely $(12),(23),(13),(45)$. By following one step of the
induction in Lemma \ref{existense of fx.prop} and carefully counting the possible cycle structures for the elements  $u^{-1}v$, one can show
that for any $n$, there cannot occur $5$ transpositions in any
metacyclic subgroup of $S_n$.
\end{exam}
In general, given a cycle structure $x$ and some $n\in \mathbb{N}$, it seems an interesting but also a hard problem to find a good bound on the number of occurrences of $x$ in an embedding of a group of cyclicity level $n$.

We can now prove Proposition \ref{series of groups}:
\begin{proof}
For two subgroups $A,B\leq \G_r$ and $n\geq r$,  denote  $$X_n(A,B) := |\{x\in\G_n|A\cap B^{x}=1\}|,$$ where $B^x$ denotes $xB x^{-1}$. Since a double coset $AxB$  splits if and only if $A\cap B^x=1$, the number $X_n(A,B)$ is the number of elements of $\G_n$ that lie in split double cosets of $A,B$ in $\G_n$.

Recall that $\mH$ was defined to be a subgroup of $\G_k$. We shall show that there is an $N$ such that for every $n\geq N$ and every metacyclic subgroup $D$ of $\G_n$, $X_n(D,\mH)>|D||\mH|$.  This will prove that
there are at least two split double cosets in $\G_n$ for $n\geq N$.

Let $T$ denote the set of all $l$-cycles in $S_{l^n}$. Then for $D\leq \G_n$:
\begin{equation*}
X_n(D,\mH) = |\G_n| - |\{x\in \G_n|D\cap \mH^{x}\not= 1\}| = |\G_n| - |\{x\in \G_n |
\alpha^{x} \in D\}|=
\end{equation*}
\begin{equation*}
=|\G_n| - |\bigcup_{\sigma \in D\cap T }^{\cdot } \{ x | \alpha^x = \sigma \}|
= |\G_n| - \sum_{\tau\in T\cap D}|\{x | \alpha^{x} =
\tau \}|.\end{equation*}

 By Corollary \ref{k-cyclicity and amount of a
fixed type lemma} there is a number $b$ for which any metacyclic
subgroup of $S_\infty$  contains at most $b$ elements of cycle
structure $p(\alpha)$. Thus,
\begin{equation}\label{collapsing point}
X_n(D,\mH) \geq |\G_n| - b\cdot |N_{\G_n}(\alpha)|.
\end{equation}
But by conditions (1) and (2):
\begin{equation}\label{subcliam to insure a boundary}
\Lim_{n\ra\infty} \frac{|\mH|d_n+b\cdot |N_{\G_n}(\alpha)|}{|\G_n|} = \Lim_{n\ra\infty} \frac{|\mH|d_n}{|\G_n|}+\frac{b}{|O_{\G_n}(\alpha)|}=0.
\end{equation}
Therefore there is an $N$ for which $|\G_n|>|\mH|d_n+b\cdot |N_{\G_n}(\alpha)|$ holds for all $n\geq N$. Using Inequality \ref{collapsing point} we obtain:
$$ X_n(D,\mH) \geq |\G_n| - b\cdot |N_{\G_n}(\alpha)|> d_n|\mH|\geq |D||\mH|$$ for every $n\geq N$ and every metacyclic subgroup $D\leq \G_n$.
\end{proof}

\subsection{Sylow subgroups of the symmetric group}\label{sylow.sec}
We use Proposition \ref{series of groups} to construct explicit examples of pairs $(\G,\mH)$
that satisfy the conditions of Corollary \ref{group_criterion.cor}:

\begin{exam}\label{sylow subgroups.exam} Let $l$ be a prime  and $n\geq 2$.
Let $\alpha_1$ be the $l$-cycle $(1\, 2 \cdots l)$, $\alpha_2$ the
product of $l$ $l$-cycles:
$$\alpha_2:=(1\,\, l+1\, \, 2l+1 \cdots (l-1)l+1)(2\,\,l+2\cdots (l-1)l+2)\cdots(l\,\,2l\,\,3l\cdots l^2).$$
For $1\leq r\leq n$, define $\alpha_r$ to be the product of
$l^{r-1}$ $l$-cycles: $$\alpha_r:=(1\,\, l^{r-1}+1\cdots
(l-1)l^{r-1}+1)\cdots (l^{r-1}\,\,2l^{r-1}\cdots l^r).$$ The group
$\G_n:=\langle \alpha_1,\ldots,\alpha_n \rangle$ is a well known
example of an $l$-Sylow subgroup of $S_{l^n}$ (see \cite[\S 5.9]{Hal}). In particular $|\G_n|=l^{\frac{l^n-1}{l-1}}.$
Let $\mH=\langle \alpha_1\rangle$. We shall prove:
\end{exam}

\begin{prop}\label{sylow.prop} If $n\geq 3$, then $S(D,\mH)>1$ for every metacyclic subgroup $D\leq \G_n$.
\end{prop}
It follows from Corollary \ref{group_criterion.cor}, that:
\begin{cor}\label{sylow.cor} For $n\geq 3$,
there is a number field $L_n$ which is Galois over $\mQ$ with $\Gal(L_n/\mQ)\cong \G_n$ such that $L_n$ and $K_n=L_n^\mH$ are equivalent by preadmissibility. \end{cor}

Furthermore we shall prove that Proposition \ref{sylow.prop} and Corollary \ref{sylow.cor} hold for $n\geq 2$ and $l\geq 5$ or $n\geq 3$.
The smallest such example therefore appears when $l=2$ and $n=3$, i.e. $\G:=\G_n=S_8(2)$ which is of order $128$.
For $l=3$ and $n=2$, $\G_n=S_9(3)=\langle (123),(147)(258)(369) \rangle$, $\mH=\langle (123)\rangle$ and Proposition \ref{sylow.prop} does not hold since $S(D,\mH)=1$ for $D=\langle (123),(456)\rangle$. For $l=2$ and $n=2$, $\G_n=S_4(2)$ is the Quaternion group which is itself metacyclic and hence  $S(D,\mH)=0$ for  $D=\G_n$.

Before proving Proposition \ref{sylow.prop}, let us find the $l$-cycles in $\G_n$. The following $l^{n-1}$ $l$-cycles are conjugates of $\alpha_1$ in $\G_n$: $$\beta_1:=(1\cdots l),\beta_2:=(l+1\cdots 2l), ...,\beta_{l^{n-1}}:=(l^{n}-l+1\cdots l^n).$$ Let $T_0=\{\beta_i|1\leq i\leq l^{n-1}\}$.
\begin{lem} \begin{enumerate} \item Any $l$-cycle in $\G_n$ is of the form $\beta_i^j$ for some $i$ and $j$.
\item $T_0$ is the set of conjugates of $\alpha_1$ in $\G_n$.
\end{enumerate}
\end{lem}
\begin{proof} \begin{enumerate} \item Assume $\gamma\in \G_n$ is another $l$-cycle that is not of this form. Without loss of generality we can assume $\gamma=(b_1\,\,b_2\cdots b_l)$ where $b_1=1$. The subgroup $\langle \alpha_1,\gamma\rangle$ is contained in an $l$-Sylow subgroup of the symmetric group on symbols $\{1,...,l,b_2,...,b_l\}$. The latter is contained in $S_{2l-1}$. An $l$-Sylow subgroup of $S_{2l-1}$ is isomorphic to $C_l$ and hence any two $l$-cycles in such group are powers of each other. Thus $\gamma=\alpha_1^j=\beta_1^j$ for some $j$, contradiction.
 \item  The group $\overline{\G}_n:=\G_n/[\G_n,\G_n]$ is isomorphic to $C_l^n$ (see \cite[\S 5.9]{Hal} and  \\ \cite[Lemma~2.11]{KSN}). As all  $\beta_i$, $1\leq i\leq l^{n-1}$ are conjugates they are mapped under the natural map $\pi:\G_n\ra \overline{\G}_n$ to the same nontrivial element. This shows that for $1\leq i\leq l^{n-1}$ and $1\leq j\leq l-1$, $\pi(\beta_i^j)=\pi(\alpha_1)$ only if $j=1$. Thus, the only conjugates of $\alpha_1$ in $\G_n$ are the elements of $T_0$.
\end{enumerate}
\end{proof}
We can now prove Proposition \ref{sylow.prop}:
\begin{proof}
Fix an $n\geq 2$ and a metacyclic subgroup $D$ of $\G_n$.  We shall calculate $X_n(D,\mH) := |\{x|D\cap \mH^{x}=\{1\}\}|$.
Any metacyclic subgroup of $\G_n$ contains at most two elements of $T_0$ and hence:
\begin{equation*}
X_n(D,\mH) = |\G_n| - |\{x\in\G_n|D\cap \mH^{x}\not= 1\}| = |\G_n| - |\{x\in \G_n |
\alpha_1^{x} \in D\}|=
\end{equation*}
\begin{equation*}\label{collapsing point2}
 =|\G_n| - |
\bigcup_{\sigma \in D\cap T_0 }^{\cdot } \{ x | \alpha_1^x = \sigma \}
|= |\G_n| - \sum_{\sigma\in D\cap T_0}|\{x | \alpha_1^{x} =
\sigma \}| \geq |\G_n| - 2\cdot |N_{\G_n}(\alpha_1)|,
\end{equation*}
where $N_{\G_n}(\alpha_1)$ denotes the normalizer of $\alpha_1$ in $\G_n$. It is of order $N_{\G_n}(\alpha_1)=\frac{|\G_n|}{|O_{\G_n}(\alpha_1)|}$. As $T_0=O_{\G_n}(\alpha_1)$ is of cardinality $l^{n-1}$, we have $X_n(D,\mH)\geq |\G_n|(1-\frac{2}{l^{n-1}})$.

The maximal order of an element in $S_{l^n}(l)$ is $l^n$ and hence the cardinality of $D$ is at most $l^n\cdot l^n=l^{2n}$. So, $|D||\mH|\leq l^{2n}\cdot l=l^{2n+1}$. Thus, in order for $X_n(D,\mH)>|D||\mH|$ to hold it suffices that:
$$ |\G_n|(1-\frac{2}{l^{n-1}}) = l^{\frac{l^n-1}{l-1}}(1-\frac{2}{l^{n-1}})> l^{2n+1}.$$

This inequality holds whenever:

(1) $n=2,l\geq 5$,

(2) $n=3,l\geq 3$ or

(3) $n\geq 4$.

This covers all cases except for $l=2,n=3$. We shall now restrict to this case. Let
$$\G:=\G_3=S_8(2)=\langle (12),(13)(24), (15)(26)(37)(48)\rangle. $$
The transpositions in $\G$ are $T_0=\{(12),(34),(56),(78)\}$.

Assume on the contrary $D\leq \G$ is a metacyclic group for which $S(D,\mH)\leq 1$. Then $X_n(D,\mH)\leq |D||\mH|$. A metacyclic subgroup of $\G$ contains at most $2$ transpositions.
We shall split the proof into two cases according to whether $D$ contains two transpositions or at most one.

Assume $D$ contains at most one transposition. Then $$X_n(D,\mH)=\{x||x^{-1}\mH x\cap D|=1\}\geq \frac{3}{4}|\G|=3\cdot 2^5.$$ But, if $|D||\mH|\geq 3\cdot 2^5$ then $|D||\mH|=2^7$ which implies $D\mH=\G$ and $|D|=2^6$. Let us show that $\G$ has no metacyclic subgroup of order $2^6$. Let $\Phi=\G^2[\G,\G]$ be the Frattini subgroup of $\G$ and let $\pi: \G\ra \G/\Phi=C_2^3$ be the natural map.  If $D$ is metacyclic of order $2^6$ then it maps under $\pi$ onto a subgroup of some $C_2^2$. As $\pi^{-1}(C_2^2)$ contains at most $2^6$ elements we must have $D = \pi^{-1}(C_2^2)$. In such a case $D$ contains $[\G,\G]$ and a transposition and hence must also contain $T_0$, contradiction.

Assume now that $D$ contains two transpositions. Then $X_n(D,\mH)=2^6$ and hence $|D|=2^6$ or $|D|=2^5$. We have seen $|D|=2^6$ cannot occur. Let us show that there is no metacyclic subgroup of $\G$ of order $2^5$ which contains two transpositions. Assume without loss of generality $(12)\in D$. There are three cases: $D\supseteq \langle (12),(34)\rangle$, $D\supseteq \langle (12),(56)\rangle$ and $D\supseteq \langle (12),(78)\rangle$.

Case $D\supseteq \langle (12),(34)\rangle$: Let $\beta_i=(2i-1,2i), 1\leq i\leq 4,$ $\tau_1=(13)(24), \tau_2=(57)(68)$ and $u=\alpha_3=(15)(26)(37)(48)$. Then $$S_8(2)= (\langle \alpha_1 \rangle \wr \langle \alpha_2 \rangle) \wr \langle \alpha_3 \rangle = (\langle \beta_1,\beta_2\rangle\rtimes \tau_1) \times (\langle \beta_3,\beta_4\rangle\rtimes \tau_2)) \rtimes \langle u\rangle.$$ Thus, any element $x\in S_8(2)$ can be written uniquely in the form $$(\prod_{i=1}^4 \beta_i^{t_i(x)}) \tau_1^{s_1(x)}\tau_2^{s_2(x)}u^{w(x)}$$ for some $t_i(x),s_1(x),s_2(x),w(x)\in \{0,1\}$ and $i=1,..,4$.  If there is an element $x\in D$ that has $w(x)=1$  then $x^{-1}(12)x$ is a transposition that is not $(12)$ nor $(34)$ leading to a contradiction. Thus $D$ can be assumed to be a subgroup of $\G_0 = (\langle \beta_1,\beta_2\rangle\rtimes \tau_1) \times (\langle \beta_3,\beta_4\rangle\rtimes \tau_2)$. Let $\Phi_0$ be the Frattini subgroup of $\G_0$ and let $\pi_0:\G_0\ra \G_0/\Phi_0=C_2^4$ be the natural map. Then 
for any subgroup $C_2^2\cong U\leq \G_0/\Phi_0$ one has $|\pi_0^{-1}(U)|\leq 2^4$. Therefore there is no metacyclic subgroup of $\G_0$ of order $2^5$.

Case $D\supseteq \langle (12),(56)\rangle$: Clearly $D\subseteq N_{\G}(\langle (12),(56)\rangle)$ but $$N_{\G}(\langle (12),(56)\rangle)=\langle \beta_1,\beta_2,\beta_3,\beta_4,u\rangle$$ is of cardinality $2^5$. Thus, if $|D|=2^5$ then $D=N_{\G}(\langle (12),(56)\rangle)$ which cannot occur since $D$ would then contain all transpositions.

Case $D\supseteq \langle (12),(78)\rangle$:  $D\subseteq N_{\G}(\langle (12),(78)\rangle)$ but $$N_{\G}(\langle (12),(78)\rangle)=\langle \beta_1,\beta_2,\beta_3,\beta_4,\tau_1\tau_2u\rangle$$ is of cardinality $2^5$. Thus if $|D|=2^5$,  $D=N_{\G}(\langle (12),(78)\rangle)$ which again cannot occur.
\end{proof}

\section{Arithmetic equivalences}\label{art.sec}
In this section we recall some characterizations of arithmetic
equivalence and local isomorphism and use them to show the
implications of Diagram (\ref{section1.1 - implications Diagram})
and prove that any other implication that holds is a composition of these implications. 
To prove the latter, we show that any implication which is not a composition of the implications of the diagram fails to hold. For this, it suffices to give examples for the non-implications: $2\not\ra 1$,
$3\not\ra 4$ and $4\not\ra 3$. The non-implication  $3\not\ra 4$
appears in Example \ref{3 not implies 4.exam}, $2\not\ra 1$
appears in \cite{Kom2} and in Example \ref{2 not implies 1.exam},
and $4\not\ra 3$ follows from Remark~\ref{4 not implies 3.exam}.

\subsection{Arithmetic equivalence}\label{art.subsec}
By \cite[Chap. 3, Theorem 1.4]{Kl}, two arithmetically equivalent number fields
  have the same $\mQ$-normal closure.
Let us therefore assume $K$ and $L$ are number fields with the same $\mQ$-normal closure
$M$ and denote $\G=\Gal(M/\mQ),\mH=\Gal(M/K)$ and $\mH'=\Gal(M/L)$.

Let $p$ be a rational prime and $v_1,\ldots,v_r$ the primes of $K$ dividing it, ordered by decreasing inertial degrees $f_i=f(v_i|p)$, $i=1,\ldots, r$. The {\it splitting type} of $p$ in $K$ is the vector $(f_1,\ldots,f_r)$.

The fields $K$ and $L$ are
arithmetically equivalent if and only if all rational primes have
the same splitting type in $K$ and $L$ (see \cite[Chap. 3, \S 1]{Kl}).
It turns out by a similar correspondence to that in Section
\ref{basic.sec} (see \cite[\S 1]{Per}) that all rational
primes have the same splitting type in $K$ and $L$ if and only if
the coset types $(C,\mH)$ and $(C,\mH')$ are the same for any
cyclic subgroup $C\leq \G$. If $\mH$ and $\mH'$ satisfy the latter
they are said to be {\it Gassmann equivalent}. It follows that if
$K$ and $L$ are arithmetically equivalent then $\mH$ and $\mH'$
are Gassmann equivalent and hence $|\mH|=|\mH'|$ and
$[K:\mQ]=[L:\mQ]$.

\begin{rem}\label{4 not implies 3.exam} The examples of Section \ref{sylow.sec} therefore show that two number fields
which are equivalent by preadmissibility (or by admissibility of
odd order groups)  can have different degrees over $\mQ$ and hence
need not be arithmetically equivalent.
\end{rem}

Note that Gassmann equivalence has another well known
characterization, namely $\mH$ and $\mH'$ are Gassmann equivalent
if and only if for any $g\in \G$:
\begin{equation*} |g^\G\cap \mH|=|g^\G\cap \mH'|, \end{equation*}
where $g^\G$ denotes the conjugacy class of $g$ in $\G$.

\subsection{Local isomorphism}

For a number field $F$, let $P(F)$ denote the set of primes of
$F$. By \cite{Iwa}, Lemma $7$, two number fields $K$ and $L$ are locally isomorphic
if and only if there is a bijection $\phi:P(K)\ra P(L)$ such that
$K_{v}\cong L_{\phi(v)}$ for every $v\in P(K)$.

It follows that two locally isomorphic number fields $K$ and $L$ (with a map $\phi$ as above) are also arithmetically
equivalent (since every $p$ has the same inertial degree in $K_v$
and $L_{\phi(v)}$) and equivalent by preadmissibility (since a
group is realizable over $K_v$ if and only if it is realizable
over $L_{\phi(v)}$). This shows the remaining implications in
Diagram (\ref{section1.1 - implications Diagram}).

In \cite{Kom2}, Komatsu gave an example of two locally isomorphic number
fields, given explicitly as radical extensions of $\mQ$, that are
not isomorphic (showing the non-implication $2\not\ra 1$).
In fact, a complete classification of locally isomorphic radical
extensions appears in \cite{JV}.

The following is a simple construction, using a different approach
from \cite{Kom2} and \cite{JV}, that assigns to every two Gassmann
equivalent subgroups of the symmetric group, two locally
isomorphic number fields. 

\begin{exam}\label{2 not implies 1.exam}
Let $M/\mQ$ be a Galois extension of number fields and $T/M$ an unramified $S_n$-extension which is Galois defined over $\mQ$, i.e. there is an $S_n$-extension $F/\mQ$ for which $T=MF$ (for such a construction see \cite{Fro}).

Let $\mH$ and $\mH'$ be two Gassmann equivalent subgroups of $S_n$
that are not conjugate in $S_n$. A method to construct such pairs
$\mH$ and $\mH'$  is given in \cite[\S 3]{Per}.

Let $\G:=\Gal(T/\mQ)\cong \Gal(T/F)\times \Gal(T/M)\cong \Gal(T/F)\times S_n$. Let us view $\mH$ identify $\mH'$ as subgroups of the latter $S_n$ and let $K=T^{\mH}$ and $L=T^{\mH'}$.
\begin{prop}\label{swallow_ram.prop}
The fields $K$ and $L$ are locally isomorphic but $K\not\cong L$.
\end{prop}
\end{exam}
\begin{proof}
As $\mH$ and $\mH'$ were chosen to be non-conjugate in $S_n$ and
as $\Gal(T/F)$ commutes with $S_n$ in $\G$, $\mH$ and
$\mH'$ are not conjugate in $\G$ and hence $K\not\cong
L$.

Let us prove that $K$ and $L$ are locally isomorphic. As $T/M$ is
unramified all primes of $M$ have cyclic decomposition groups. Let
$v$ be a prime of $M$ and let $C$ be a cyclic subgroup of $S_n$ such
that the  decomposition group of $v$ in $T/M$ is (the conjugacy
class of) $C$.  By the correspondence in Section \ref{basic.sec},
there is a bijection between the primes $v_1,\ldots,v_r$ (resp.
$w_1,\ldots,w_s$) of $K$ (resp. of $L$) that divide $v$ and the
double cosets $Cx_1\mH,\ldots,Cx_r\mH$ (resp.
$Cy_1\mH',\ldots,Cy_s\mH'$) such that
$[K_{v_i}:M_v]=\frac{|Cx_i\mH|}{|\mH|}$ (resp.
$[L_{w_i}:M_v]=\frac{|Cy_i\mH'|}{|\mH'|}$). As $\mH$ and $\mH'$
are Gassmann equivalent in $S_n$  the coset type $(C,\mH)$ is the
same as the coset type $(C,\mH')$. Thus $r=s$, $|\mH|=|\mH'|$ and
one has:
\begin{equation*}
d_i:=[K_{v_i}:M_v]=\frac{|Cx_i\mH|}{|\mH|}=\frac{|Cy_i\mH'|}{|\mH'|}=[L_{w_i}:M_v]
\end{equation*} for $i=1,2,\ldots,r$.
But as $T/M$ is unramified and $M_v$ has a unique unramified
extension of degree $d_i$, one has $K_{v_i}\cong L_{w_i}$, for
$i=1,\ldots, r$. This establishes a bijection $\phi$ for any prime
$v$ of $M$ between the prime divisors of $v$ in $K$ and those in
$L$ such that $K_{v_K}\cong L_{\phi(v_K)}$ for any prime $v_K$ of $K$
dividing $v$. We therefore obtain a bijection $\phi:P(K)\ra P(L)$
such that $K_{v_K}\cong L_{\phi(v_K)}$ for any prime $v_K$ of $K$.
Thus, $K$ and $L$ are locally isomorphic.
\end{proof}
Given a $\G$-extension $F/\mQ$, the process of creating an extension $M/\mQ$
for which $MF/M$ is an unramified $\G$-extension is often called swallowing ramification or Abhyankar's Lemma.

\subsection{Preadmissibility and arithmetic equivalence}

To show $3\not\ra 4$, we use an example from \cite{Kom1} of two arithmetically equivalent
fields $K$ and $L$ that are not locally isomorphic and show that these $K$ and $L$ are in fact
inequivalent by preadmissibility.

\begin{exam}\label{3 not implies 4.exam} Let $K=\mQ(\sqrt[32]{m})$ and $L=\mQ(\sqrt{2}\sqrt[32]{m})$ where $m\not=\pm 1,\pm 2$ is a square free integer that satisfies $m\equiv 1$ (mod $2^7$). In \cite[Lemma 4]{Kom1}, Komatsu shows that $K$ and $L$ are arithmetically equivalent. Let us show that $K$ and $L$ are not equivalent by preadmissibility and nor by admissibility. Since $m\equiv 1$ (mod $2^7$) there is a unit $u\in \mathbb{Z}_2$ for which $u^{32}=m$. So, the polynomials that define $K$ and $L$ factor over $\mQ_2$ into irreducible factors as follows:
$$ x^{32}-m= (x-u)(x+u)(x^2+u^2)(x^4+u^4)(x^8+u^8)(x^{16}+u^{16})$$
$$ x^{32}-2^{16}m= (x^2-2u)(x^2+2u)(x^2-2ux+2u)(x^2+2ux+2u)(x^8+16u^8)(x^{16}+2^8u^{16}). $$
Note that by \cite[Lemma 10]{Kom1} all the above factors are irreducible in $\mQ_2[x]$.

Therefore, there are $6$ primes $v_1,\ldots,v_6$ in $K$ and $6$ primes $w_1,\ldots,w_6$ in $L$ that divide $2$. Let us
assume the primes are ordered so that $[K_{v_i}:\mQ_2]\geq
[K_{v_{i+1}}:\mQ_2]$ and \\ $[L_{w_i}:\mQ_2]\geq [L_{w_{i+1}}:\mQ_2]$
for  $i=1,\ldots,5$. Considering the above factorizations we have:
$K_{v_1}=\mQ_2(\mu_{32})$, $K_{v_2}=\mQ_2(\mu_{16})$,
$L_{w_1}=\mQ_2(\sqrt[16]{-2})$, $L_{w_2}=\mQ_2(\sqrt[8]{-2})$ and
$[K_{v_i}:\mQ_2]\leq 4$, $[L_{w_i}:\mQ_2]\leq 4$ for $i\geq 3$.

Let $A:=C_{16}^{10}$. By local class field theory the maximal
abelian extension $k_{ab}$ of a $p$-adic field $k$ has Galois
group $\Gal(k_{ab}/k)$ that is isomorphic to the profinite
completion of the group $k^*$ (see e.g. \cite[Chap. 14, \S 6]{Ser3}). Therefore the maximal abelian
extension $K_{v_2,ab,16}$ of exponent $16$ of $K_{v_2}$ has Galois
group $\Gal(K_{v_2,ab,16}/K_{v_2})\cong C_{16}^{10}=A$. Similarly,
$\Gal(K_{v_1,ab,16}/K_{v_1})\cong C_{16}^{18}$,
$\Gal(L_{w_1,ab,16}/L_{w_1})\cong C_{16}^{17}\times C_2$,
$\Gal(L_{w_2,ab,16}/L_{w_2})\cong C_{16}^{9}\times C_2$ and
$\emph{rk}(\Gal(K_{v_i,ab,16}/K_{v_i})),\emph{rk}(\Gal(L_{w_i,ab,16}/L_{w_i}))\leq
6$ for $i\geq 3$. We can now see that $A$ is realizable over two
completions of $K$ and only one completion of $L$. Thus, $A$ is
$K$-preadmissible but not $L$-preadmissible and $K$ and $L$ are
inequivalent by preadmissibility. It also follows that $A$ is
not $L$-admissible but by \cite{Nef} $A$ is
$K$-admissible (as the primes $v_1,v_2$ are evenly even). Thus,
$K$ and $L$ are  inequivalent by admissibility.
\end{exam}

\end{document}